\documentclass[11pt]{amsart}
\usepackage{geometry}                
\usepackage{graphicx}
\usepackage{amssymb}
\usepackage{amsmath}
\usepackage{amsthm}
\usepackage{mathtools}
\usepackage{hyperref}
\usepackage{mathrsfs}
\usepackage{epstopdf}
\usepackage{tikz}

\usepackage{comment}

\allowdisplaybreaks

\newtheorem{prop}{Proposition}
\newtheorem{cor}[prop]{Corollary}
\newtheorem{lem}[prop]{Lemma}
\newtheorem{thm}[prop]{Theorem}
\theoremstyle{definition}

\newcommand*{\NN}{\mathbb{N}}

\newcommand*{\RR}{\mathbb{R}}

\newcommand*{\PP}{\mathbb{P}}
\newcommand*{\EE}{\mathbb{E}}

\newcommand*{\1}{\mathbf{1}}

\newcommand*{\scrH}{\mathscr{H}}

\newcommand*{\st}{\,:\,}

\newcommand*{\rootvert}{\rho}
\newcommand*{\parent}[1]{\pi(#1)}

\newcommand*{\DeltaKR}{\widetilde\Delta}

\newcommand*{\G}{\mathrm{G}} 

\DeclareMathOperator{\Prob}{Prob}
\DeclareMathOperator{\Var}{Var}

\DeclareMathOperator{\Lip}{Lip}

\DeclareMathOperator{\br}{br}
\DeclareMathOperator{\bargr}{\overline{gr}}
\DeclareMathOperator{\maxgr}{maxgr}

\DeclareMathOperator{\arctanh}{arctanh}

\newcommand*{\wg}{\mathrm{w}} 

\DeclarePairedDelimiter{\abs}{\lvert}{\rvert}

\DeclarePairedDelimiter{\norm}{\|}{\|}
\DeclarePairedDelimiterX{\inprod}[2]{\langle}{\rangle}{#1,\ #2}

\title{Concentration of Broadcast Models on Trees}
\author{Christopher Shriver}
\date{\today}                                           

\begin{document}

\begin{abstract}
	An inequality of K. Marton \cite{marton1996} shows that the joint distribution of a Markov chain with uniformly contracting transition kernels exhibits concentration. We prove an analogous inequality for broadcast models on finite trees. We use this inequality to develop a condition for the sequence of depth-$k$ marginals of a broadcast model on a rooted infinite tree to form a normal L\'evy family in terms of the Lipschitz constants of the transition kernels and the growth rate of the tree.
\end{abstract}

\maketitle

\tableofcontents

\section{Introduction and Main Results}

Let $\scrH$ be a Polish metric space of diameter at most 1. We give the product space $\scrH^n$ the normalized Hamming metric
	\[ d(x,y) \coloneqq \frac{1}{n} \sum_{i=1}^n d(x_i, y_i) .\]
We denote the space of Borel probability measures on $\scrH^n$ by $\Prob(\scrH^n)$ and define on this space the transportation metric
	\[ \bar{d}(\mu,\nu) = \inf_{\lambda} \int_{(\scrH^n)^2} d(x,y)\, d\lambda(x,y) \]
where the infimum is over couplings of $\mu$ and $\nu$ (see \cite{dudley2004} Section 11.8). We also define the relative entropy between $\mu,\nu \in \Prob(\scrH^n)$ by
	\[ D(\mu \| \nu) = \int \log \frac{d\mu}{d\nu}\, d\mu \]
when $\mu \ll \nu$; otherwise we set $D(\mu \| \nu) = +\infty$ (see \cite{cover2006} for the discrete case or \cite{dembo2010a} Appendix D.3 for the continuous version used here).

A concentration inequality due to McDiarmid (Theorem 3.1 in \cite{mcdiarmid1998}) implies that if $\nu$ is a product probability measure on $\scrH^n$ then for any 1-Lipschitz $f \colon \scrH^n \to \RR$ with $\int f\, d\nu = 0$ we have
	\[ \nu \{ f > \lambda \} \leq e^{-2 \lambda^2 n} \quad \text{for all } \lambda > 0 . \]
This can be viewed as a quantitative refinement of the weak law of large numbers.
Later, Marton \cite{marton1996} showed that if $\nu \in \Prob(\scrH^n)$ is the joint distribution of a Markov chain $(X_1, \ldots, X_n)$ taking values in $\scrH$ then
	\[ \bar{d}(\mu,\nu) \leq \frac{1}{a} \sqrt{\frac{1}{2n} D(\mu \| \nu)} \quad \text{for all } \mu \in \Prob(\scrH^n) \]
where $a \in (0,1]$ is a measure of contractivity of the Markov kernels.

These results have very different statements and proofs; McDiarmid's inequality is proven by bounding the exponential moments $\int e^{\lambda f}\, d\nu$, while Marton's proof uses a coupling argument and makes no mention of Lipschitz functions. However, a later result of Bobkov and G\"otze (Theorem 1.3 of \cite{bobkov1999}) shows that a transportation-entropy inequality such as Marton's is equivalent to an exponential moment bound of the form used to prove McDiarmid's inequality; therefore in retrospect we can view Marton's inequality as a generalization of McDiarmid's.

In the present paper we generalize this inequality further to broadcast models on trees. Here, an $\scrH$-valued broadcast model indexed by a finite rooted tree $T = (V,E)$ is a family of $\scrH$-valued random variables $(X_v)_{v \in V}$ such that if $v \in V$ is a vertex with children $w_1, \ldots, w_k$ then $X_{w_1}, \ldots, X_{w_k}$ are conditionally independent given $X_v$ and their distributions are determined by the value of $X_v$. 
We call the joint distribution $\nu \in \Prob(\scrH^V)$ a Markov measure indexed by $T$. See Section \ref{sec:chains} for a more precise definition.

Broadcast models are natural models for processes such as communications networks or phylogenetic trees, where information originates at a root node and is distributed from each node to its children with some probability of error. The term ``broadcast model'' is also often used to refer to the special case where $\scrH = \{0,1\}$ and each ``bit'' $X_v$ is equal to the bit at the parent of $v$ with probability $1-p$ and is equal to the opposite bit with probability $p$. In the present paper we refer to this case as the Ising model; see Section \ref{sec:ising}.

Our main result is stated precisely as follows:
	
\begin{thm}
\label{thm:main}
	Let $T = (V,E)$ be a tree with $n$ vertices, and let $\nu$ be a $\scrH$-valued Markov measure indexed by $T$ with $b$-Lipschitz transition kernels.
	Then for any 1-Lipschitz $f \colon \scrH^V \to \RR$ with $\int f\, d\nu = 0$ we have
		\[ \int e^{n \lambda f} \, d\nu \leq e^{\lambda^2 \Delta^2 /8}, \]
	where $\Delta$ is a function of $b$ and $T$ defined in Section \ref{sec:delta} below. Equivalently,
		\[ \bar{d}(\mu,\nu) \leq \frac{\Delta}{n} \sqrt{\frac{1}{2} D(\mu \| \nu) }\]
	for all $\mu \in \Prob(\scrH^V)$.
\end{thm}


The tail bound resulting from Theorem \ref{thm:main} via the exponential moment method is
	\[ \nu \big\{ \abs[\big]{f - \int f\, d\nu} > \varepsilon \big\} \leq 2 e^{-2 n^2 \varepsilon^2 / \Delta^2} \quad \forall f \in \Lip_1 \big( \scrH^{V} \big) . \]

McDiarmid and Marton's inequalities can be recovered as special cases of Theorem \ref{thm:main}: If each vertex of $T$ has at most $d$ children and $bd<1$ then
	\[ \Delta \leq \frac{\sqrt{n}}{1-bd}. \tag{$\dagger$} \]
See Proposition \ref{prop:altDeltabound} for a proof. If $d=1$ then we recover Marton's inequality (her parameter $a$ is equal to $1-b$). If $\nu$ is a product measure then we can take $b = 0$, and we recover McDiarmid's inequality.

Kontorovich (\cite{kontorovich2012}, Theorem 8) has also obtained a Lipschitz function exponential moment bound for Markov measures indexed by trees by bounding what are called $\eta$-mixing coefficients of the process in terms of $b$ and the width of the tree $T$. Earlier work of Kontorovich and Ramanan \cite{kontorovich2008} and (independently) Chazottes, Collet, K\"ulske, and Redig \cite{chazottes2006} showed that concentration is controlled by various norms of a matrix whose entries are the $\eta$-mixing coefficients. These $\eta$-mixing coefficients, however, were defined with linear-time (as opposed to tree-indexed) processes in mind; in order to make sense of $\eta$-mixing coefficients in this context, Kontorovich interprets a tree-indexed process as a linear-time one by fixing a breadth-first ordering of the vertices. In the present paper we control the exponential moments by making more direct use of the tree structure. In Section \ref{sec:comparison} we discuss in more detail the relationship between these results and Theorem \ref{thm:main}, in particular whether they are sufficient to establish Theorem \ref{thm:growthrate}. \\

We use Theorem \ref{thm:main} to study the rate of concentration of sequences of tree-indexed Markov measures. We say that a sequence of metric probability spaces $(X_k, d_k, \mu_k)$ is a \emph{L\'evy family} if for every $\varepsilon>0$ we have
	\[ \sup \left\{ \mu_k \{ f > \varepsilon \} \st f \in \Lip_1(X_k,d_k),\ \int f\, d\mu_k = 0 \right\} \to 0 \quad \text{as } k \to \infty. \]
If $X_k = \scrH^{n_k}$ for some sequence $1 \leq n_1 < n_2 < \cdots$, we say the sequence of metric probability spaces is a \emph{normal} L\'evy family if for each $\varepsilon>0$ there exist postive constants $c_1,c_2$ such that the supremum is bounded above by $c_1 e^{-c_2 n_k \varepsilon^2}$. This terminology is used similarly in Ledoux's book \cite{ledoux2001}; we use a slightly different concentration function and allow for an arbitrary sequence of dimensions $n_k$.


Note that McDiarmid's inequality implies that if $\{\mu_k\}_{k \in \NN}$ is any sequence of product measures with $\mu_k \in \Prob(\scrH^k)$, then the sequence $(\scrH^k, d_k, \mu_k)$ is a normal L\'evy family. 

Theorem \ref{thm:main} has the following consequence for the concentration of a sequence of Markov measures indexed by trees:

\begin{cor}
\label{cor:conc}
	If $\{T_k = (V_k, E_k) \st k \in \NN\}$ is a sequence of finite trees with $b$-Lipschitz Markov measures $\nu_k$ and corresponding $\Delta_k$ then the sequence of metric probability spaces $\{(\scrH^{V_k}, d, \nu_k) \st k \in \NN \}$ (with $d$ the Hamming metric) is a L\'evy family if $\Delta_k = o(\abs{V_k})$ and is a normal L\'evy family if $\Delta_k = O(\sqrt{\abs{V_k}})$.
\end{cor}

One natural way of producing a sequence of finite trees is to start with an infinite but locally finite rooted tree $T = (V,E)$ and for each $k$ let $T_k$ be the subtree induced by the set of vertices within distance $k$ of the root. The inequality ($\dagger$) shows that if each vertex of $T$ has at most $d$ children then $\Delta_k = O(\sqrt{\abs{V_k}})$ as long as $bd<1$. We show that this asymptotic, and hence being a normal L\'evy family, holds for a wider range of $b$, and replace the degree bound $d$ with more precise measures of the growth rate of $T$ which are defined below (if every vertex of $T$ has exactly $d$ children then all relevant measures are equal to $d$):

\begin{thm}
\label{thm:growthrate}
	If $b<1$ and $T$ has bounded degree then $\Delta_k = o(\abs{V_k})$. 
	
	If $b^2 \maxgr T < 1$ then $\Delta_k = O(\sqrt{\abs{V_k}})$, and if $b^2 \bargr T > 1$ then $\Delta_k \ne O(\sqrt{\abs{V_k}})$.
	
	In particular, if $T$ is subperiodic then $\maxgr T = \bargr T$ so this gives the exact location of a phase transition in the growth rate of $\Delta_k$.
\end{thm}

A definition of subperiodicity is given below; see also \cite{lyons2016}. Every regular tree is subperiodic. We also note that if $T$ is subperiodic then $\bargr T$ is equal to the branching factor $\br T$, which determines phase transitions related to percolation and random walks \cite{lyons1990}, reconstruction for the binary symmetric channel \cite{evans2000,peres2003}, and uniqueness and extremality of the free boundary Gibbs state for the Ising model (see Section 2.2 of \cite{evans2000} for a brief survey). In general, the branching factor is bounded above by both $\bargr$ and $\maxgr$.

In the final section we turn to examining the special case of the Ising model. Using this example we show that Theorem \ref{thm:main} is close to sharp in the following sense:

\begin{thm}
\label{thm:optimality}
	Let $T = (V,E)$ be a finite tree with $n$ vertices and $b \in [0,1)$, and suppose $C = C(T,b) \in \RR$ is such that for all metric spaces $\scrH$ of diameter at most 1 and for each $b$-Lipschitz $\scrH$-valued Markov measure $\nu$ indexed by $T$ we have
		\[ \int e^{n \lambda f}\, d\nu \leq e^{C^2 \lambda^2 / 8} \]
	for all 1-Lipschitz $f \colon \scrH^V \to \RR$ with mean 0, or equivalently
		\[ \bar{d}(\mu,\nu) \leq \frac{C}{n} \sqrt{\frac{1}{2}D (\mu \| \nu)} \]
	for all $\nu \in \Prob(\scrH^V)$. Then
		\[ C \geq \Delta \frac{\sqrt{1-b^2}}{2} . \]
\end{thm}

In general, Theorem \ref{thm:growthrate} only gives conditions for whether analysis of the growth rate of $\Delta_k$ can or cannot establish that a sequence is a (normal) L\'evy family; it is possible for a sequence of measures to be a normal L\'evy family even if $\Delta_k \ne O(\sqrt{\abs{V_k}})$. The problem is that $\Delta_k$ depends on the Lipschitz constants of the transition kernels only through their maximum, which can be affected by a single atypical kernel. We show that in the case of the Ising model with uniform transition probabilities, where the maximum is much more representative of the overall behavior of the process, a phase transition actually occurs in the quality of concentration:

\begin{thm}
\label{thm:bsctransition}
	For $p \in (0,1/2]$, the sequence of depth-$k$ marginals of the Ising model with transition probability $p$ on an infinite tree $T$ is not a normal L\'evy family if $\Delta_k \ne O(\sqrt{\abs{V_k}})$.
	
	In particular, it is not a normal L\'evy family if $b^2 \bargr T > 1$, so that if $T$ is such that $\maxgr T = \bargr T$ then we have a phase transition at this location.
\end{thm}

The regularity of the tree $T$ does not affect the existence of this transition, only our present ability to state its location in terms of natural quantities; see the definition of $\G(T)$ and subsequent discussion in Section \ref{subsec:refine} below.

Some numerical evidence suggests that the phase transition may occur at $b^2 \maxgr T = 1$. In Figure \ref{fig:31growthrate} we plot $\Delta_k^2/\abs{V_k}$ as a function of $k$ for various values of $b$, with $T$ the ``3-1 tree'' defined in Section \ref{subsec:trees} and pictured in Figure \ref{fig:31tree}. It seems that $\Delta_k^2/\abs{V_k}$ is concave down for $b < \frac{1}{\sqrt{\maxgr T}} = \frac{1}{\sqrt{3}}$ and concave up for larger $b$. For comparison, we also include the same plot with $T$ the binary tree, where a transition is known to occur at $b = 1/\sqrt{2}$. The lack of symmetry in the 3-1 tree makes $\Delta_k$ much more difficult to calculate efficiently compared to the binary tree; this is why the depth only goes up to 25. The images were produced using Matplotlib \cite{hunter2007}, and calculations for the 3-1 tree were done in part using NumPy \cite{vanderwalt2011}.
\begin{figure}
\centering
\includegraphics[width=0.48\textwidth]{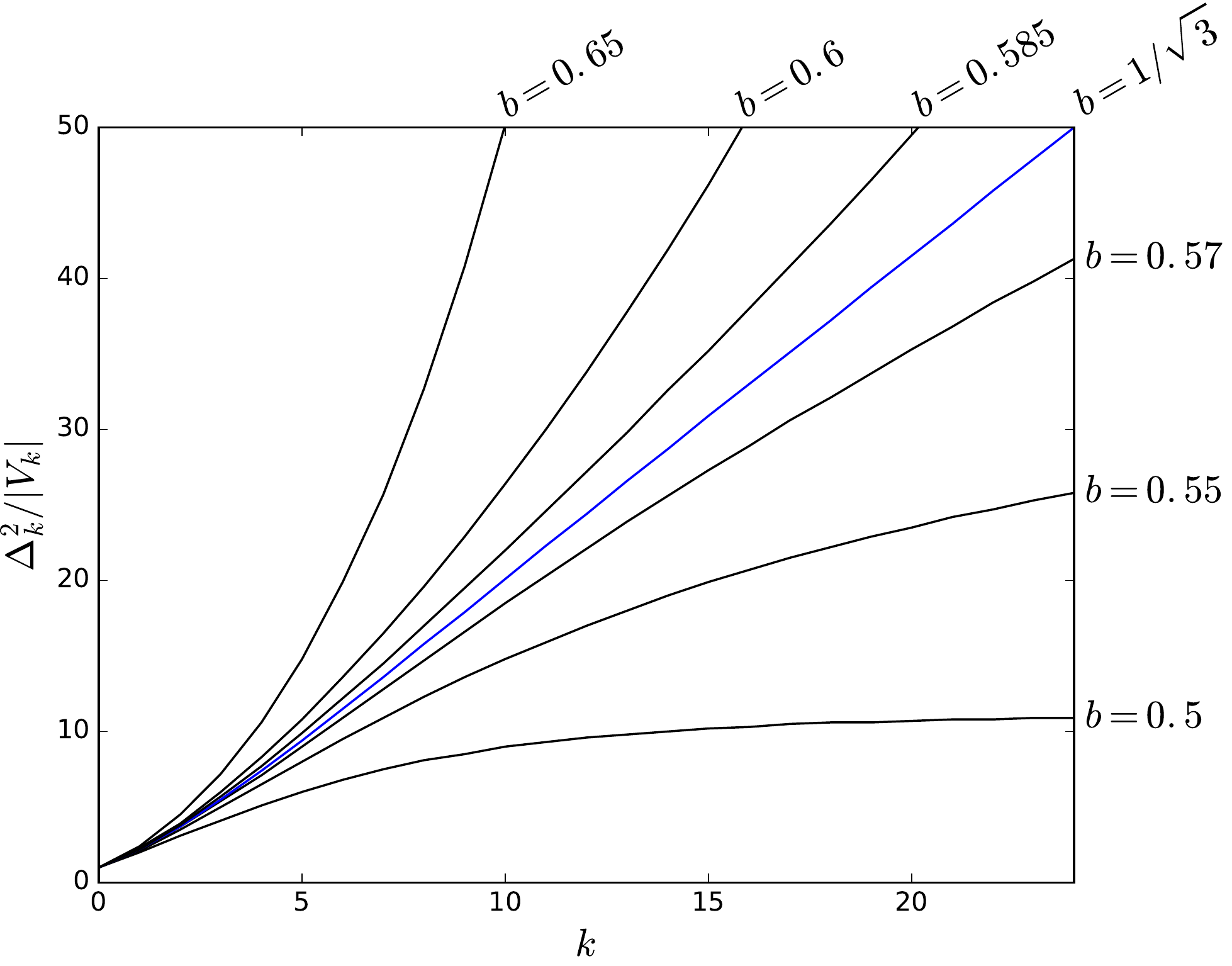}
\quad
\includegraphics[width=0.48\textwidth]{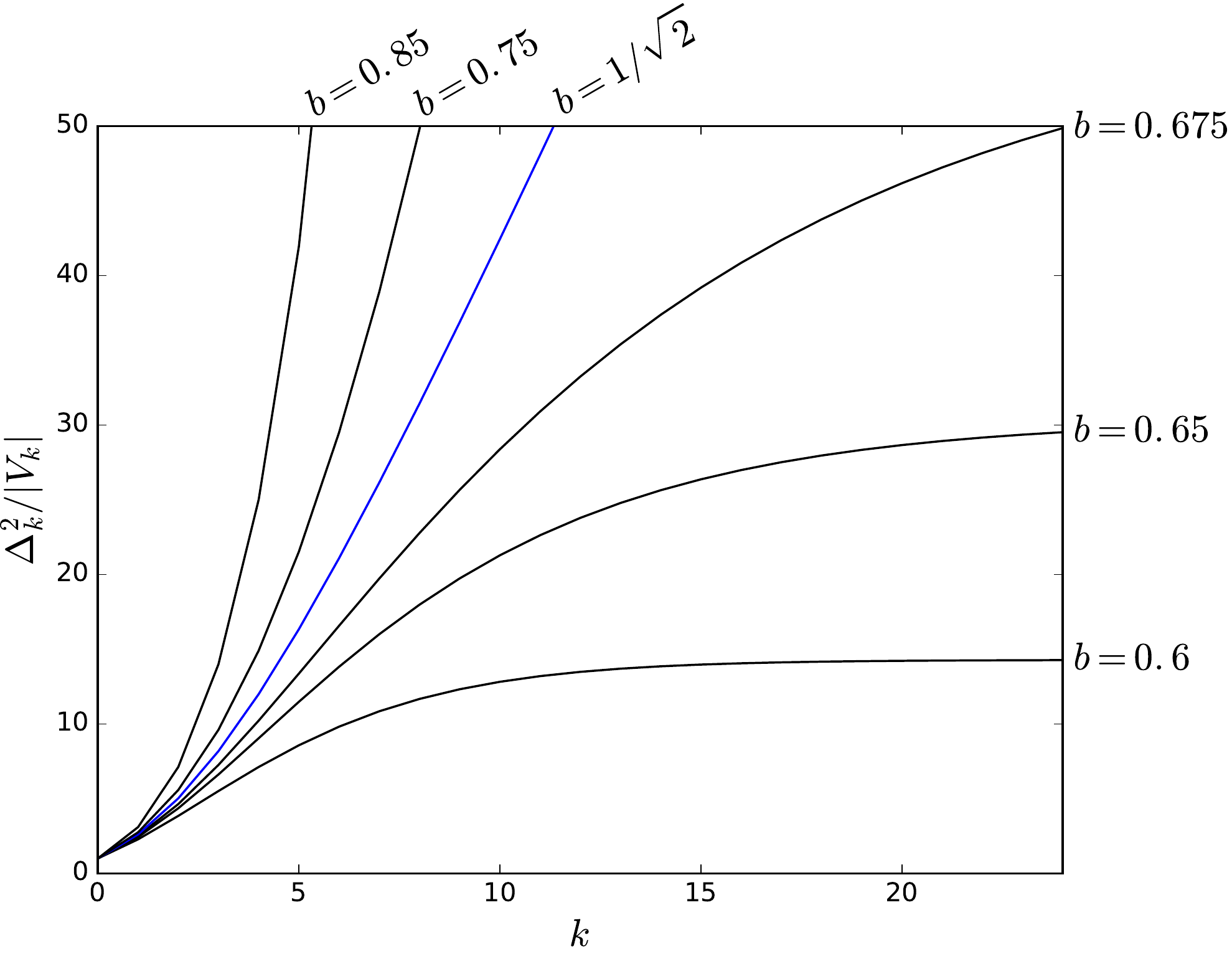}
\caption{Comparison of growth rate of $\Delta_k$ for different values of $b$, where on the left $T$ is the ``3-1 tree'' defined in Section \ref{subsec:trees} and pictured in Figure \ref{fig:31tree}, and on the right $T$ is the binary tree. The conjectured critical value on the left is $b = 1/\sqrt{3}$, and the known critical value on the right is $b = 1/\sqrt{2}$.}
\label{fig:31growthrate}
\end{figure}

\subsection{Directions for further work}
\subsubsection{Refinement of Theorem \ref{thm:growthrate}}
\label{subsec:refine}
Theorem \ref{thm:growthrate} is inconclusive for $b$ in the interval $\big[ (\maxgr T)^{-1/2}, (\bargr T)^{-1/2} \big]$, which has positive length for general trees.

One way to resolve this is to define a measure of tree growth specifically designed to determine whether $\Delta_k$ has the desired growth rate: Given an infinite tree $T$, define $\G(T)$ by the formula
	\[ \big(\G(T)\big)^{-1/2} = \sup \left\{ b \in \RR \st \limsup_{k \to \infty} \frac{1}{\abs{V_k}} \sum_{(v,w) \in V_k^2} b^{d(v,w)} < \infty \right\} \]
or equivalently
	\[ \G(T) = \inf \left\{ \lambda > 0 \st \limsup_{k \to \infty} \frac{1}{\abs{V_k}} \sum_{(v,w) \in V_k^2} \lambda^{-d(v,w)/2} < \infty \right\} . \]
This takes a similar form to the formulas
	\[ \br T = \inf \left\{ \lambda > 0 \st \liminf_{\Pi \to \infty} \sum_{v \in \Pi} \lambda^{-d(\rootvert, v)} < \infty \right\} \]
(where $\Pi$ are `cutsets'; see \cite{lyons1990}) and
	\[ \bargr T = \inf \left\{ \lambda > 0 \st \limsup_{k \to \infty} \sum_{v \in L_k} \lambda^{-d(\rootvert,v)} < \infty \right\} \]
(where $L_k$ is the set of vertices at distance $k$ from the root $\rootvert$; this is equivalent to the definition given in Section \ref{subsec:trees}).
By Proposition \ref{prop:Deltabound}, $\Delta_k = O(\sqrt{V_k})$ if $b^2 \G(T) < 1$ and $\Delta_k \ne O(\sqrt{V_k})$ if $b^2 \G(T) > 1$. Therefore Theorem \ref{thm:growthrate} implies
	\[ \bargr T \leq \G(T) \leq \maxgr T. \]
In particular, $\G(T) = d$ if $T$ is the $d$-ary tree, and $\G(T) \leq d$ if all vertices have at most $d$ children.

Based on the preceding comparisons it seems reasonable to interpret $\G(T)$ as a measure of the growth rate of $T$. However, compared to $\maxgr T$, which is both the spectral radius of the adjacency matrix of $T$ and a slight variant of $\bargr T$, the definition of $\G(T)$ is less natural; for this reason we have chosen not to express the main results of this paper in terms of $\G(T)$.

Is it true that $\G(T) = \maxgr T$? Figure \ref{fig:31growthrate} suggests that this may be true for at least one tree with $\bargr T \ne \maxgr T$.

\subsubsection{Generalizations of Theorem \ref{thm:main}}
The method used to prove Theorem \ref{thm:main} may also establish concentration for Bayesian networks, which are like broadcast models with multiple root/source nodes.

A similar approach may also yield concentration for the marginal on the leaf nodes.

\subsection{Overview}
Section 2 contains definitions and auxiliary results used to proved the main theorems. In Section 3 we prove the main result on concentration of Lipschitz functions, and in Section 4 we prove the conditions stated in Theorem \ref{thm:growthrate} which guarantee concentration of a sequence of tree-indexed Markov measures. In Section 5 we restrict to the special case of the Ising model, and use it to prove that Theorem \ref{thm:main} is almost sharp. We also use the Ising model to compare our results to related work in Section 6.

\subsection{Acknowledgements}
This material is based upon work supported by the National Science Foundation under Grant No.~DMS 1344970.

The author would like to thank Tim Austin and Georg Menz for many helpful discussions and feedback on earlier versions of the paper.

\section{Definitions and Lemmas}

\subsection{Tree notation}
\label{subsec:trees}
We write $T = (V,E)$ to mean that $T$ is a rooted tree with vertex set $V$ and edge set $E$. The root vertex is denoted by $\rootvert$. We always assume $T$ to be \emph{locally finite}, i.e. each vertex has finite degree. We consider $V$ to be a metric space, with the distance between two vertices given by the number of edges in the unique simple path between them.

It will often be useful to endow $V$ with the following natural partial order: we say $v \leq w$ if $v$ lies on the path from the root to $w$ (including $v = \rootvert$ or $v = w$). In this situation we also say that $v$ is an ancestor of $w$ or that $w$ is a descendant of $v$. Then every pair of vertices has a well-defined meet $v \wedge w$, which is the unique maximal vertex which is an ancestor of both $v$ and $w$. The vertex $v \wedge w$ can also be characterized as the place where the paths from $\rootvert$ to $v$ and from $\rootvert$ to $w$ diverge.

The \emph{parent} of a vertex $v$ is the unique maximal ancestor of $v$ which is not equal to $v$. We denote this vertex by $\parent{v}$; note that the root is the only vertex with no parent. The \emph{children} of $v$ are those vertices in the set $\pi^{-1}(v) = \{ w \in V \st \pi(w) = v\}$. A vertex with no children is called a \emph{leaf}.

We denote the set of descendants in the $r$th generation after $v$ by
	\[ D_r(v) \coloneqq \{ w \in V \st v \leq w,\ d(v,w) = r\} = (\pi^r)^{-1}(v). \]
The upper growth rate and the maximum local growth rate of $T$ are defined by
	\[ \bargr T \coloneqq \limsup_{r \to \infty} \abs{D_r(\rootvert)}^{1/r} \]
and
	\[ \maxgr T \coloneqq \limsup_{r \to \infty} \max_{v \in V} \abs{D_r(v)}^{1/r} . \]
The notion of upper growth rate is well-known; see for example \cite{lyons2016} Section 3.3.

Note that $\bargr T \leq \maxgr T$. To see that the inequality may be strict, consider the tree (taken from \cite{lyons1990}) defined as follows: for each $k \geq 0$ the $k$th level $D_k(\rootvert)$ consists of $2^k$ vertices. Both vertices of $D_1(\rootvert)$ are children of $\rootvert$. Choose an ordering for each level, and for $k>0$ and $1 \leq j \leq 2^{k-1}$ take the set of children of the $j$th vertex in $D_k(\rootvert)$ to be the $j$th group of 3 vertices of $D_{k+1}(\rootvert)$. The remaining $2^{k-1}$ vertices of $D_k(\rootvert)$ each get one child, chosen in order from the remaining $2^{k-1}$ vertices of $D_{k+1}(\rootvert)$. See Figure \ref{fig:31tree}. Clearly $\bargr T = 2$, but $T$ contains arbitrarily deep 3-ary subtrees so $\maxgr T = 3$.
\begin{figure}
\begin{center}
\begin{tikzpicture}[
		level 1/.style={sibling distance=7em}, 
		level 2/.style={sibling distance=5em},
		level 3/.style={sibling distance=1.6em},
		level 4/.style={sibling distance=0.5em},
		level 5/.style={sibling distance=0.15em}, level distance = 0.75cm
	]
	\coordinate (root) {} [fill] circle (1.5pt)
		child { [fill] circle (1pt)
			child { [fill] circle (1pt)
				child { [fill] circle (1pt)
					child { [fill] circle (1pt) child child child }
					child { [fill] circle (1pt) child child child }
					child { [fill] circle (1pt) child child child }
					}
				child {
					[fill] circle (1pt)
					child { [fill] circle (1pt) child child child }
					child { [fill] circle (1pt) child child child }
					child { [fill] circle (1pt) child child child }
					}
				child {
					[fill] circle (1pt)
					child { [fill] circle (1pt) child child child }
					child { [fill] circle (1pt) child child child }
					child { [fill] circle (1pt) child }
					}
				}
			child { [fill] circle (1pt)
				child { [fill] circle (1pt)
					child { [fill] circle (1pt) child }
					child { [fill] circle (1pt) child }
					child { [fill] circle (1pt) child }
					}
				child { [fill] circle (1pt)
					child { [fill] circle (1pt) child }
					}
				child { [fill] circle (1pt)
					child { [fill] circle (1pt) child }
					}
				}
			child { [fill] circle (1pt)
				child { [fill] circle (1pt)
					child { [fill] circle (1pt) child }
					}
				}
			}
		child { [fill] circle (1pt)
			child { [fill] circle (1pt)
				child { [fill] circle (1pt)
					child { [fill] circle (1pt) child }
					}
				}
			}
	;
\end{tikzpicture}
\caption{The first few levels of the 3-1 tree, which satisfies $\bargr T = 2$ and $\maxgr T = 3$.}
\label{fig:31tree}
\end{center}
\end{figure}
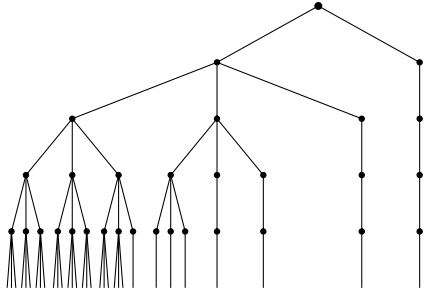

The following equivalent characterization of $\bargr T$ is sometimes more convenient than the definition above:

\begin{prop}
\label{prop:altgrowthformula}
For any tree $T$,
	\[ \bargr T = \limsup_{r \to \infty} \abs{B_r(\rootvert)}^{1/r}, \]
where $B_r(\rootvert)$ is the closed ball of radius $r$ around the root.
\end{prop}
\begin{proof}
	Since $D_r(\rootvert) \subseteq B_r(\rootvert)$, we clearly have $\bargr T \leq \limsup_{r \to \infty} \abs{B_r(\rootvert)}^{1/r}$.
	
	For the converse inequality, note that
		\[ \abs*{B_r(\rootvert)} = \sum_{k=0}^r \abs{D_k(\rootvert)} \leq (r+1) \max_{0 \leq k \leq r} \abs{D_k(\rootvert)} , \]
	so that
		\[ \limsup_{r \to \infty} \abs{B_r(\rootvert)}^{1/r} \leq \limsup_{r \to \infty} \max_{0 \leq k \leq r} \abs{D_k(\rootvert)}^{1/r}. \]
	For any fixed $K$ we have $\limsup_{r \to \infty} \max_{0 \leq k < K} \abs{D_k(\rootvert)}^{1/r} = 1$, so
		\[ \limsup_{r \to \infty} \abs{B_r(\rootvert)}^{1/r} \leq \limsup_{r \to \infty} \max_{K \leq k \leq r} \abs{D_k(\rootvert)}^{1/r} \leq \limsup_{r \to \infty} \max_{K \leq k \leq r} \abs{D_k(\rootvert)}^{1/k} = \sup_{K \leq k} \abs{D_k(\rootvert)}^{1/k} .\]
	Taking the infimum over $K$ finishes the proof.
\end{proof}

Given $v \in V$, let $T_v$ denote the subgraph induced by the set of descendants of $v$, considered as a tree rooted at $v$. A tree $T$ is called \emph{subperiodic} if there exists $R \geq 0$ such that for any $v \in V$ there exists $w \in V$ with $d(\rootvert,w) \leq R$ such that $T_v$ is isomorphic to a subtree of $T_w$ (see \cite{lyons2016}).

\begin{prop}
	If $T$ is subperiodic then $\bargr T = \maxgr T$.
\end{prop}

\begin{proof}
	We just need to check that $\bargr T \geq \maxgr T$, since the converse holds for general trees. Let $R \geq 0$ be as given by the definition of subperiodicity and for each $r$ let $v(r)$ be a vertex maximizing $\abs{D_r(v)}$; note we can take $d(\rootvert,v(r)) \leq R$. Then $D_r(v(r)) \subseteq B_{r+R}(\rootvert)$, so by Proposition \ref{prop:altgrowthformula}
		\[ \maxgr T = \limsup_{r \to \infty} \abs{D_r(v(r))}^{1/r} \leq \limsup_{r \to \infty} \abs{B_{r+R}(\rootvert)}^{1/r} = \limsup_{r \to \infty} \abs{B_r(\rootvert)}^{1/r} = \bargr T. \qedhere \]
\end{proof}

\subsection{Markov measures indexed by trees}
\label{sec:chains}
%

We essentially follow the definitions from \cite{benjamini1994}, but with some natural modifications to allow for a continuous state space.

Let $T = (V,E)$ be a finite tree of depth $r$ and $\scrH$ be a Polish metric space of diameter at most 1. A Markov measure indexed by $T$ is a measure $\nu \in \Prob(\scrH^V)$ given in terms of its marginal $\nu_\rootvert$ at the root and a collection of probability kernels $\{ q_v \colon \scrH \to \Prob(\scrH) \st v \in V \setminus \{\rootvert\}\}$ as follows: for Borel $A \subseteq \scrH^V$,
	\[ \nu(A) = \int_{\scrH^{L_0}} \int_{\scrH^{L_1}} \cdots \int_{\scrH^{L_r}} \1_A(y) \prod_{v \in L_k} q_v(dy_v | y_{\pi(v)}) \cdots \prod_{v \in L_1} q_v(dy_v | y_\rootvert)\, \nu_\rootvert(dy_\rootvert) . \]
If $\scrH$ is at most countable and $A = \{a\}$ for $a \in \scrH^V$ then this reduces to the standard formula
	\[ \nu(\{a\}) = \nu_\rootvert(\{a_\rootvert\}) \prod_{v \in V \setminus \{\rootvert\}} q_v(\{a_v\}|a_{\pi(v)}) . \]

If each kernel is $b$-Lipschitz as a map $(\scrH,d) \to (\Prob(\scrH),\bar{d})$ then we say that the Markov measure $\nu$ is $b$-Lipschitz. 

\subsection{Descendant generating function}
\label{sec:delta}
In this section we define the parameter $\Delta$ which appears in the bounds of the main results of this paper.

Fix $b \in [0,1)$ and a finite rooted tree $T$. We define a function $\delta \colon V \to \RR$ by setting
	\[ \delta(v) = \sum_{r = 0}^\infty \abs*{D_r(v)}\, b^r . \]
We also define
	\[ \Delta = \norm{\delta}_{\ell^2(V)} = \left( \sum_{v \in V} \delta(v)^2 \right)^{1/2} . \]
Because of its coefficients as a power series in $b$ we refer to $\delta(v)$ as the \emph{descendant generating function at $v$}. For infinite trees we could consider $\delta(v)$ to be a formal power series or restrict $b$ to be smaller than the radius of convergence $(\limsup\abs*{D_r(v)}^{1/r})^{-1} = (\bargr T_v)^{-1}$, but since $T$ is finite only finitely many terms of the series are nonzero. To simplify notation, the dependence on $b$ will always be kept implicit; in context $b$ will typically be the Lipschitz constant of the relevant Markov measure.

We have the following equivalent characterization of $\delta$:
\begin{lem}
	If $v \in V$ is a leaf then $\delta(v)=1$. If $v$ is not a leaf, $\delta$ satisfies the recurrence
		\[ \delta(v) = 1 + b \sum_{w \st \pi(w) = v} \delta(w) . \]
\end{lem}
\begin{proof}
That $\delta(v) = 1$ on leaves follows immediately from the fact that a leaf has no descendants other than itself.

For the recurrence, note that for $v \in V$ and $r > 0$ we have
	\[ \abs*{D_r(v)} = \sum_{w \st \pi(w) = v} \abs*{D_{r-1}(w)} . \]
Therefore
	\begin{align*}
		1+b \sum_{w \st \pi(w) = v} \delta(w) &= 1+b \sum_{w \st \pi(w) = v} \sum_{r = 0}^\infty \abs*{D_r(w)}\, b^r \\
		&= \abs*{D_0(v)}\, b^0 + \sum_{r=0}^\infty \sum_{w \st \pi(w) = v} \abs*{D_r(w)}\, b^{r+1} \\
		&= \abs*{D_0(v)}\, b^0 + \sum_{r=0}^\infty \abs*{D_{r+1}(v)}\, b^{r+1} \\
		&= \delta(v). \qedhere
	\end{align*}
\end{proof}

In some situations below we will be interested in the descendant generating functions of finite subtrees of some fixed infinite tree. The following lemma will be useful:

\begin{lem}
\label{lem:deltaQ}
	Let $T = (V,E)$ be an infinite rooted tree and let $T' = (V',E')$ be a finite subtree with the same root. Let $Q \colon \RR^V \to \RR^V$ be given by $Qf(v) = \sum_{w \st \pi(w) = v} f(w)$; then the descendant generating function of $T'$ is given by
		\[ \delta'(v) = \left( \sum_{j=0}^\infty (bQ)^j \1_{V'} \right) (v) \quad \text{for } v \in V'. \]
\end{lem}
\begin{proof}
	Let
		\[ D_j'(v) \coloneqq \{ w \in V' \st v \leq w \text{ and } d(v,w) = j \} \]
	denote the set of descendants of $v$ in the $j$th generation; then
		\[ \delta'(v) = \sum_{j=0}^k \abs*{D_j'(v)} b^j . \]
	Note that we can truncate the sum at some finite $k$ since $T'$ is finite.
	
	Now we show that $\abs*{D_j'(v)} = Q^j \1_{V'}(v)$. This is clear for $j=0$, and the remaining cases follow by induction: assuming $\abs*{D_{j-1}'(v)} = Q^{j-1} \1_{V'}(v)$, we have
	\begin{align*}
		\abs*{D_j'(v)} &= \sum_{w \st \pi(w) = v} \abs*{D_{j-1}'(w)} \\
		&= \sum_{w \st \pi(w) = v} Q^{j-1} \1_{V'}(w) \\
		&= Q^{j-1} \sum_{w \st \pi(w) = v} \1_{V'}(w) \\
		&= Q^{j-1} \left[ Q \1_{V'}(v) \right] \\
		&= Q^j \1_{V'}(v).
	\end{align*}
	Therefore 
		\[ \delta'(v) = \sum_{j=0}^k [Q^j \1_{V'}(v)] b^j = \left( \sum_{j=0}^k (bQ)^j \right) \1_{V'}(v) . \qedhere \]
\end{proof}

The same operator $Q$ is considered in \cite{lyons1990}. There, the relevant quantity is the branching factor $\br T$, which turns out to be the radius of the point spectrum of $Q$; the following operator norm calculation (which will also be useful below) along with Gelfand's formula implies that $\sqrt{\maxgr T}$ is the spectral radius of $Q$.

\begin{prop}
\label{prop:Qopnorm}
	\[ \norm{Q^j} = \sqrt{\max_v \abs{D_j(v)}} . \]
\end{prop}
\begin{proof}
For $w \ne \rootvert$, let $\pi(w)$ denote the parent of $w$. We claim that the adjoint $Q^*$ is given by
		\[ Q^* f(w) = \left\{ \begin{array}{ll} f(\pi(w)), & w \ne \rootvert \\ 0, & w = \rootvert . \end{array} \right. \]
	To see this, we check that for every $f,g \in \ell^2(V)$ we have
	\begin{align*}
		\inprod{Qf}{g} &= \sum_{v \in V} Qf(v) \overline{g(v)} \\
		&= \sum_{v \in V} \sum_{w \st \pi(w) = v} f(w) \overline{g(\pi(w))} \\
		&= \sum_{w \in V} \sum_{v \st \pi(w) = v} f(w) \overline{g(\pi(w))} \\
		&= \sum_{w \in V} f(w) \overline{g(\pi(w)) \1_{w \ne \rootvert}}.
	\end{align*}
	
	Now for any $f \in \ell^2(V)$ we have
	\begin{align*}
		\norm{(Q^*)^j f}_2 &= \left( \sum_{w \in V \setminus \{\rootvert\}} \big( f(\pi^j(w)) \big)^2 \right)^{1/2} \\
		&= \left( \sum_{v \in V} \abs{\{ w \in V  \st \pi^j(w) = v\}} f(v)^2 \right)^{1/2} \\
		&= \left( \sum_{v \in V} \abs{D_j(v)} f(v)^2 \right)^{1/2} \\
		&\leq \sqrt{\max_{v\in V} \abs{D_j(v)}}\, \norm{f}_2,
	\end{align*}
	so $\norm{Q^j} = \norm{(Q^*)^j} \leq \sqrt{\max_v \abs{D_j(v)}}$. If $w \in V$ is such that $\abs{D_j(w)}$ is maximal then
		\[ \norm*{Q^j \frac{\1_{D_j(w)}}{\sqrt{\abs{D_j(w)}}}} = \norm*{\frac{\abs{D_j(w)} \1_{\{w\}}}{\sqrt{\abs{D_j(w)}}}} = \sqrt{\abs{D_j(w)}} = \sqrt{\max_v \abs{D_j(v)}}. \]
	Therefore in fact $\norm{Q^j} = \sqrt{\max_v \abs{D_j(v)}}$
\end{proof}
	
The following is used in the proof of Theorem \ref{thm:growthrate} to estimate $\Delta$; in particular it implies that changing the root affects $\Delta$ by at most a factor of $\sqrt{1-b^2}$, independent of any properties of the tree:

\begin{prop}
\label{prop:Deltabound}
For any tree $T$ and $b \in [0,1)$,
	\[ \sum_{(w_1,w_2) \in V^2} b^{d(w_1,w_2)} \leq \Delta^2  \leq \frac{1}{1-b^2} \sum_{(w_1,w_2) \in V^2} b^{d(w_1,w_2)} . \]
\end{prop}

\begin{proof}
By definition of $\delta(v)$, we have
	\[ \delta(v)^2 = \sum_{r=0}^\infty \abs{\{ (w_1, w_2) \in V^2 \st v \leq w_1 \wedge w_2,\ d(w_1, v) + d(w_2, v) = r \}} \cdot b^r \]
and hence
\begin{align*}
	\Delta^2 = \sum_{v \in V} \delta(v)^2 &= \sum_{(w_1,w_2) \in V^2} \sum_{v \leq w_1 \wedge w_2} b^{d(w_1, v) + d(w_2, v)} \\
	&= \sum_{(w_1,w_2) \in V^2} \sum_{v \leq w_1 \wedge w_2} b^{d(w_1, w_2) + 2 d(v,w_1 \wedge w_2) } \\
	&= \sum_{(w_1,w_2) \in V^2} b^{d(w_1,w_2)} \sum_{v \leq w_1 \wedge w_2} (b^2)^{d(v,w_1 \wedge w_2) } . \\
\end{align*}
Bounding the inner sum above by $\frac{1}{1-b^2}$ gives the upper bound, and bounding it below by 1 gives the lower bound.
\end{proof}

We also have the following bound on $\Delta$ mentioned in the introduction which, while less precise and broadly applicable than Proposition $\ref{prop:Deltabound}$ and not used below, has the advantage of a simpler dependence on the tree structure:

\begin{prop}
\label{prop:altDeltabound}
	If each vertex of $T$ has at most $d$ children and $bd<1$ then
	\[ \Delta \leq \frac{\sqrt{n}}{1-bd}. \]
\end{prop}

\begin{proof}
	The bound on the number of children of each vertex gives $\abs{D_r(w)} \leq d^r$ for any $w \in V$ and $r \in \NN$; therefore, using that $bd<1$, for any $w \in V$ we have
		\[ \delta(w) = \sum_{r=0}^\infty \abs{D_r(w)}\, b^r \leq \sum_{r=0}^\infty d^r b^r = \frac{1}{1-bd} . \]
	Hence
		\[ \Delta = \left( \sum_{w \in V} \delta(w)^2 \right)^{1/2} \leq \left( n \left( \frac{1}{1-bd} \right)^2 \right)^{1/2} = \frac{\sqrt{n}}{1-bd} . \qedhere \]
\end{proof}

\subsection{Other Lemmas}

In the proof of Theorem \ref{thm:main} below we establish the exponential moment bound directly; the transportation-entropy inequality follows from the following equivalence due to Bobkov and G\"otze:
\begin{thm}[Theorem 1.3 from \cite{bobkov1999}]
	Let $(\Omega,d)$ be a bounded metric space and $\nu \in \Prob(\Omega)$. Then $\nu$ satisfies
		\[ \bar{d}(\mu,\nu) \leq C \sqrt{D(\mu \| \nu)} \]
	for all $\mu \in \Prob(\Omega)$
	if and only if
		\[ \int e^{\lambda f}\, d\nu \leq e^{C^2 \lambda^2/4} \]
	for all 1-Lipschitz $f$ with $\int f\, d\nu = 0$.
\end{thm}

Within the proof of Theorem \ref{thm:main} we will use weighted Hamming metrics on product spaces: if $I$ is a finite index set and $\wg \colon I \to \RR^{>0}$ is a positive function on $I$ then we define a metric $d_\wg$ on $\scrH^I$ by
	\[ d_\wg(x,y) \coloneqq \frac{1}{\abs{I}} \sum_{i \in I} \wg(i)\, d(x_i, y_i) . \]
	
The resulting transportation metric $\overline{d_\wg}$ on $\Prob(\scrH^I)$ satisfies the following formula for product measures:

\begin{lem}
\label{lem:dbarproduct}
	For each $i \in I$ let $\mu_i, \nu_i \in \Prob(\scrH)$. For any positive weight function $\wg \colon I \to \RR^{>0}$ we have
		\[ \overline{d_\wg}\left(\bigtimes_{i \in I} \mu_i,\ \bigtimes_{i \in I} \nu_i \right) = \frac{1}{\abs{I}}\sum_{i\in I} \wg(i)\, \bar{d}(\mu_i, \nu_i). \]
\end{lem}
\begin{proof}
	For each $i$, let $\lambda_i \in \Prob(\scrH^2)$ be a coupling of $\mu_i$ and $\nu_i$.
	Then $\bigtimes_{i \in I} \lambda_i$ is a coupling of $\bigtimes_{i \in I} \mu_i$ and $\bigtimes_{i \in I} \nu_i$, so
	\begin{align*}
		\overline{d_\wg}\left(\bigtimes_{i \in I} \mu_i,\ \bigtimes_{i \in I} \nu_i \right) &\leq \int_{(\scrH^2)^I} \frac{1}{\abs{I}} \sum_{i \in I} \wg(i)\, d(x_i, y_i) \prod_{i \in I} \lambda_i(dx_i, dy_i) \\
		&= \frac{1}{\abs{I}} \sum_{i \in I} \wg(i) \int_{\scrH^2} d(x_i, y_i) \lambda_i(dx_i, dy_i) .
	\end{align*}
	Taking the infimum over the $\lambda_i$'s gives
		\[ \overline{d_\wg}\left(\bigtimes_{i \in I} \mu_i,\ \bigtimes_{i \in I} \nu_i \right) \leq \frac{1}{\abs{I}} \sum_{i \in I} \wg(i) \bar{d}(\mu_i,\nu_i) . \]
	
	Conversely, let $\lambda$ be any coupling of $\bigtimes_{i \in I} \mu_i$ and $\bigtimes_{i \in I} \nu_i$, and let $\lambda_i$ be its marginals. Then each $\lambda_i$ is a coupling of $\mu_i$ and $\nu_i$, so
	\begin{align*}	
		\int_{(\scrH^2)^I} \left[ \frac{1}{\abs{I}} \sum_{i \in I} \wg(i)\, d(x_i, y_i) \right] \lambda(dx, dy) &= \frac{1}{\abs{I}} \sum_{i \in I} \left[ \int_{(\scrH^2)^I} \wg(i)\, d(x_i, y_i)\, \lambda(dx, dy) \right] \\
		&= \frac{1}{\abs{I}} \sum_{i \in I} \wg(i) \left[ \int_{\scrH^2} d(x_i, y_i)\, \lambda_i(dx_i, dy_i) \right] \\
		&\geq \frac{1}{\abs{I}} \sum_{i \in I} \wg(i)\, \bar{d}(\mu_i, \nu_i).
	\end{align*}
	Taking the infimum over all couplings $\lambda$ completes the proof.
\end{proof}

The following inequality due to Hoeffding is the foundation of the exponential moment bound in Theorem \ref{thm:main}. It is essentially used in \cite{hoeffding1963} but appears more explicitly (with proof) as Lemma 2.6 in McDiarmid's survey \cite{mcdiarmid1998}.

\begin{lem}
\label{lem:hoeffding}
	Let $(\Omega,\mu)$ be a probability space and let $f \colon \Omega \to \RR$ satisfy $\int f\, d\mu = 0$ and $\sup_{x,y \in \Omega} \abs{f(x) - f(y)} \leq L$. Then for any $\lambda \geq 0$
		\[ \int e^{\lambda f(x)} \mu(dx) \leq e^{\lambda^2 L^2/8} . \]
\end{lem}

We will also need the following version of McDiarmid's inequality:

\begin{prop}
\label{prop:inhomog}
	Let $n \in \NN$ and let $\wg \colon \{1, 2, \ldots, n\} \to \RR^{>0}$ be a weight function. Given $p \in \Prob(\scrH)$ denote the product measure by $p^n \in \Prob(\scrH^n)$. For any 1-Lipschitz $f\colon (\scrH^n,d_\wg) \to \RR$ with $\int f dp^n = 0$ we have
		\[ \int e^{n \lambda f} \, dp^n \leq e^{\lambda^2 \sum_{i=1}^n \wg(i)^2 /8}. \]
\end{prop}

\begin{proof}
	We induct on $n$.
	
	The case $n=1$ follows from Hoeffding's Lemma (Lemma \ref{lem:hoeffding}).
	
	For the inductive step, assume the result for $n-1$. Let $g \colon \scrH^{n-1} \to \RR$ be given by
		\[ g(y) = \int_\scrH f(y,x)\, dp(x). \]
	Then, using that $f$ is 1-Lipschitz on $\scrH^n$,
		\[ \abs{g(y) - g(y')} \leq \int_\scrH \abs{f(y,x)- f(y',x)}\, dp(x) \leq \frac{1}{n} \sum_{i = 1}^{n-1} \wg(i)\, d(y_i,y'_i) = \frac{n-1}{n} d_{\{\wg(1), \ldots, \wg(n-1)\}}(y,y'). \]
	Setting $L = \frac{n-1}{n}$, this shows that $g/L$ is 1-Lipschitz on its domain $(\scrH^{n-1}, d_{\{\wg(1),\ldots,\wg(n-1)\}})$. By the inductive hypothesis,
		\[ \int e^{n \lambda g} \, dp^{n-1} = \int d^{(n-1) \lambda (g/L)} \, dp^{n-1} \leq e^{\lambda^2 \sum_{i=1}^{n-1} \wg(i)^2 /8}. \]
	Finally, using the above and the case $n=1$ (noting that for each $y \in \scrH^{n-1}$ the function $x \mapsto \frac{n}{\wg(n)} (f(y,x) -g(y))$ is 1-Lipschitz with mean zero),
	\begin{align*}
		\int e^{n \lambda f} \, dp^n &= \int \left[ \int e^{n \lambda(f(y,x) - g(y))}\, dp(x)\right] e^{n \lambda g(y)}\, dp^{n-1}(y) \\
		&\leq \int \left[e^{\lambda^2 \wg(n)^2 /8} \right] e^{n \lambda g(y)}\, dp^{n-1}(y) \\
		&\leq e^{\lambda^2 \sum_{i=1}^n \wg(i)^2/8}. \qedhere
	\end{align*}
\end{proof}

\section{Proof of Theorem \ref{thm:main}}
Let $T$ be a fixed finite tree of depth $r$, and for $0 \leq k \leq r$ let $T_k$ be the subtree induced by $V_k$, the set of vertices of distance at most $k$ from the root. Let $L_k$ denote the leaves of $T_k$, i.e. the vertices of $T$ of distance exactly $k$ from the root.
	Throughout, $\delta$ refers to the descendant generating function of the original tree $T$.
	
	We prove the following statement by induction on $k$:
	\begin{quote}
		For every $k$, if $\nu$ is a $b$-Lipschitz $\scrH$-valued Markov measure indexed by $T_k$ and $f \colon \scrH^{V_k} \to \RR$ is 1-Lipschitz with respect to the Hamming metric with weights $\wg \colon V_k \to \RR$ given by
		\[ \wg(v) = \left\{ \begin{array}{ll} \delta(v), & v \in L_k \\ 1, & v \not\in L_k , \end{array} \right. \]
	then
		\[ \int e^{\abs{V_k} \lambda f}\, d\nu \leq e^{\lambda^2 \sum_{v \in V_k} \delta(v)^2/8}. \]
	\end{quote}
	Note that for the final case $k=r$ we have $\wg \equiv 1$, so $d_\wg$ is the standard (unweighted) Hamming metric.
	
	The base case $k=0$ follows from Hoeffding's Lemma (Lemma \ref{lem:hoeffding}): assuming $f \colon \{\rootvert\} \to \RR$ is $\delta(\rootvert)$-Lipschitz, we get
		\[ \int e^{1 \lambda f} \, d\nu_\rootvert \leq e^{\lambda^2 \delta(\rootvert)^2/8} . \]
	
	For the inductive step, assume that $1 \leq k \leq r$ and that the result holds for $T_{k-1}$. Considering $\scrH^{V_k} \cong \scrH^{V_{k-1}} \times \scrH^{L_k}$, define $g \colon \scrH^{V_{k-1}} \to \RR$ by
		\[ g(y) = \int_{\scrH^{L_k}} f(y, x) \prod_{w\in L_k} q_{w}(dx_w | y_v). \]
	Letting $\nu_{V_{k-1}}$ denote the marginal of $\nu$ on $V_{k-1}$, note that $\int g\, d\nu_{V_{k-1}} = \int f\, d\nu = 0$.
	
	We now consider whether $g$ is a Hamming Lipschitz function. Let $\wg \colon V_k \to \RR$ be as defined above and let $\wg\vert_{L_k}$ denote its restriction to $L_k$. Then, by Monge-Kantorovich-Rubinstein duality, since for each fixed $y \in \scrH^{V_{k-1}}$ the function $x \mapsto \frac{n}{\abs{L_k}} f(y,x)$ is 1-Lipschitz from $(\scrH^{L_k}, d_{\wg\vert_{L_k}})$ to $\RR$, for each $y,y' \in \scrH^{V_{k-1}}$ we have
	\begin{align*}
		\abs{g(y) - g(y')} &= \abs*{\int_{\scrH^{L_k}} f(y, x) \prod_{w \in L_k} q_{w}(dx_w | y_{\pi(w)}) - \int_{\scrH^{L_k}} f(y', x) \prod_{w \in L_k} q_{w}(dx_w | y'_{\pi(w)}) } \\[0.3cm]
		&\leq \abs*{\int_{\scrH^{L_k}} f(y, x) \prod_{w \in L_k} q_{w}(dx_w | y_{\pi(w)}) - \int_{\scrH^{L_k}} f(y, x) \prod_{w \in L_k} q_{w}(dx_w | y'_{\pi(w)}) } \\
		&\qquad + \abs*{\int_{\scrH^{L_k}} \big[ f(y, x) - f(y', x) \big] \prod_{w \in L_k} q_{w}(dx_w | y'_{\pi(w)})} \\[0.3cm]
		&\leq \frac{\abs{L_k}}{n} \overline{d_{\wg\vert_{L_k}}} \left( \bigtimes_{w \in L_k} q_{w}(\cdot | y_{\pi(w)}) ,\  \bigtimes_{w \in L_k} q_{w}(\cdot | y'_{\pi(w)}) \right) \\
		&\qquad + \int_{\scrH^{L_k}} \abs*{ f(y, x) - f(y', x) } \prod_{w \in L_k} q_{w}(dx_w | y'_{\pi(w)}) \\[0.3cm]
		&= \frac{1}{n} \sum_{w\in L_k} \wg(w) \bar{d} \big(  q_{w}(\cdot | y_{\pi(w)}) ,\ q_{w}(\cdot | y'_{\pi(w)}) \big) \\
		&\qquad + \int_{\scrH^{L_k}} \abs*{ f(y, x) - f(y', x) } \prod_{w \in L_k} q_{w}(dx_w | y'_{\pi(w)})
	\end{align*}
	where the last equality uses Lemma \ref{lem:dbarproduct} above. We can bound the integrand of the second term using the Lipschitz assumption on the function $f\colon \scrH^{V_k} \to \RR$; in particular $f$ is $\frac{\delta(v)}{n}$-Lipschitz on each vertex $v \in L_k$ and $\frac{1}{n}$-Lipschitz on $V_k \setminus L_k = V_{k-1}$. Substituting also $\wg(w) = \delta(w)$ for $w \in L_k$, this gives
	\begin{align*}
		\abs{g(y) - g(y')}  &\leq \frac{1}{n} \sum_{w \in L_k} \delta(w) \bar{d} \big(  q_{w}(\cdot | y_{\pi(w)}) ,\ q_{w}(\cdot | y'_{\pi(w)}) \big) + \frac{1}{n}\sum_{v \in V_{k-1}} d(y_v, y'_v) .
	\end{align*}
	Now using the Lipschitz assumption on the Markov kernels, the first term is bounded by $\frac{1}{n}\sum_{w \in L_k} \delta(w) b\, d(y_{\pi(w)}, y'_{\pi(w)})$. If we write this sum as a double sum, grouping vertices $w$ with the same parent $v \in L_{k-1}$, we get
	\begin{align*}
		\abs{g(y) - g(y')} &\leq \frac{1}{n} \sum_{v \in L_{k-1}} \sum_{w \st \pi(w) = v} \delta(w) b\, d(y_v, y'_v) + \frac{1}{n}\sum_{w \in V_{k-1}} d(y_w, y'_w)  \\
		&= \frac{1}{n} \sum_{v \in L_{k-1}} \left(1 + b\sum_{w \st \pi(w) = v} \delta(w) \right) d(y_v, y'_v) + \frac{1}{n}\sum_{v \in V_{k-1}\setminus L_{k-1}} d(y_v, y'_v) \\
		&= \frac{1}{n} \sum_{v \in L_{k-1}} \delta(v) d(y_v, y'_v) + \frac{1}{n}\sum_{v \in V_{k-1}\setminus L_{k-1}} d(y_v, y'_v)
	\end{align*}
	Therefore if we let $L = \frac{\abs{V_{k-1}}}{n}$ we can apply the inductive hypothesis to $g/L$: Letting $\nu_{V_{k-1}}$ be the marginal of $\nu$ on $V_{k-1}$, which is a Markov measure indexed by $T_{k-1}$,
		\[ \int e^{n \lambda g} \, d\nu_{V_{k-1}} = \int e^{\abs{V_{k-1}} \lambda [g/L]} \, d\nu_{V_{k-1}} \leq e^{\lambda^2 \sum_{w \in V_{k-1}} \delta(w)^2 /8}. \]
	
	To finish the proof, for fixed $y \in \scrH^{V_{k-1}}$ we apply Proposition \ref{prop:inhomog} to the 1-Lipschitz, expectation-zero function $(\scrH^{L_k},d_{\wg\vert_{L_k}}) \ni x \mapsto \frac{n}{\abs{L_k}}[f(y,x) - g(y)]$. By definition of the Markov measure $\nu$ we get
	\begin{align*}
		\int_{\scrH^{V_k}} e^{n \lambda f} \, d\nu &= \int_{\scrH^{V_{k-1}}} \int_{\scrH^{L_k}} e^{n \lambda f(y,x)} \prod_{w \in L_k} q_{w}(dx_w | y_{\pi(w)})\, \nu_{V_{k-1}}(dy) \\
		&= \int \left[ \int e^{\abs{L_k} \lambda \cdot \frac{n}{\abs{L_k}}[f(y,x) - g(y)]}\, \prod_{w \in L_k} q_{w}(dx_w | y_{\pi(w)}) \right] e^{n\lambda g(y)} \, \nu_{V_{k-1}}(dy) \\
		&\leq \int \left[ e^{\lambda^2 \sum_{w \in L_k} \delta(w)^2 /8} \right] e^{n\lambda g(y)} \, \nu_{V_{k-1}}(dy) \\
		&\leq \left[ e^{\lambda^2 \sum_{w \in L_k} \delta(w)^2 /8} \right] e^{\lambda^2 \sum_{w \in V_{k-1}} \delta(w)^2 /8} \\
		&= e^{\lambda^2 \sum_{w \in V_k} \delta(w)^2 /8}. \qedhere
	\end{align*}

\section{Proofs of Concentration Results}

\subsection{Corollary \ref{cor:conc}}
Theorem \ref{thm:main} combined with a standard application of the exponential moment method gives that 
each Markov measure $\nu_k$ on $\scrH^{V_k}$ satisfies, for any $\varepsilon>0$ and $f \in \Lip_1(\scrH^{V_k} \big)$,
		\[ \sup \left\{ \nu_k \big\{ \abs[\big]{f - \int f\, d\nu_k} > \varepsilon \big\} \st f \in \Lip_1 \big( \scrH^{V_k} \big) \right\} \leq 2 e^{-2 \abs{V_k}^2 \varepsilon^2 / \Delta_k^2} . \]
	If $\Delta_k = o(\abs{V_k})$ then the right-hand side goes to zero as $\abs{V_k} \to \infty$, and $\Delta_k = O(\sqrt{\abs{V_k}})$ will ensure that it does so exponentially fast. \\
	
\subsection{Theorem \ref{thm:growthrate}}
\label{subsec:thmgrowthratepf}
The first part uses the upper bound in Proposition \ref{prop:Deltabound}:  Suppose every vertex has degree at most $d$. For each $r \in \NN$, let $C_r$ be the number of vertices in the ball of radius $r$ centered at a vertex in the infinite $d$-regular tree. Then for any $w \in V_k$ we have $\abs{B_r(w)} \leq C_r$, so since $b<1$
	\begin{align*}
	\frac{1}{\abs{V_k}^2} \sum_{(w_1,w_2) \in V_k^2} b^{d(w_1,w_2)} &= \frac{1}{\abs{V_k}^2}\sum_{w_1 \in V_k}\left( \sum_{w_2 \in B_r(w_1)} b^{d(w_1,w_2)} + \sum_{w_2 \not\in B_r(w_1)} b^{d(w_1,w_2)} \right) \\
	&\leq  \frac{1}{\abs{V_k}^2}\sum_{w_1 \in V_k}\left( 1 \cdot C_r + b^r \cdot \abs{V_k} \right) \\
	&= \frac{C_r}{\abs{V_k}} + b^r .
\end{align*}
Since we have assumed that the full tree $T$ is infinite, we have $\lim_{k \to \infty} \abs{V_k} = \infty$ and hence
	\[ \limsup_{k \to \infty} \frac{1}{\abs{V_k}^2} \sum_{(w_1,w_2) \in V_k^2} b^{d(w_1,w_2)} \leq b^r. \]
Since $r$ was arbitrary and $b<1$,
	\[ \lim_{k \to \infty} \frac{1}{\abs{V_k}^2} \sum_{(w_1,w_2) \in V_k^2} b^{d(w_1,w_2)} = 0. \]
This shows that the right-hand side of Proposition \ref{prop:Deltabound} is $o(\abs{V_k}^2)$, which implies that $\Delta_k = o(\abs{V_k})$. \\

Now suppose $b^2 \maxgr T < 1$; we show that $\Delta_k = O(\sqrt{\abs{V_k}})$.
By Lemma \ref{lem:deltaQ},
		\[ \Delta_k = \norm*{ \left( \sum_{j=0}^k (bQ)^j \right) \1_{V_k}}_2 , \]
	so by the triangle inequality and definition of the operator norm
		\[ \frac{\Delta_k}{\sqrt{\abs{V_k}}} = \norm*{\left( \sum_{j=0}^k (bQ)^j \right) \frac{\1_{V_k} }{\sqrt{\abs{V_k}}} }_2 \leq \sum_{j=0}^k b^j \norm{Q^j}_2 . \tag{$\dagger$} \]
	Proposition \ref{prop:Qopnorm} implies that
		\[ b \limsup_{j \to \infty} \norm{Q^j}_2^{1/j}  = \sqrt{b^2 \maxgr T} < 1 , \]
	so the series on the right converges as $k \to \infty$, and hence $\Delta_k = O(\sqrt{\abs{V_k}})$.
	
It seems possible that one could be able to replace $\maxgr T$ in this result by some smaller quantity, maybe even $\bargr T$, through more careful analysis. Specifically, using the $\ell^2$ operator norm in ($\dagger$) may not be optimal since we only need to bound functions of the form $Q^j \1_{V_k}$. See Figure \ref{fig:31growthrate} and the relevant discussion in the introduction, however, for some evidence that $\maxgr T$ is actually appropriate. 

One might also ask whether the application of the triangle inequality in ($\dagger$) shares some blame for the appearance of $\maxgr T$ rather than some smaller quantity, but the following shows that this is the best we can hope for using the operator norm:
\begin{prop}
\label{prop:opnormbad}
	If $b^2 \maxgr T > 1$, then $\lim_{k \to \infty} \norm*{\sum_{j=0}^k (bQ)^j}_2 = \infty$.
\end{prop}

\begin{proof}
	Pick $\varepsilon>0$ small enough that $b^2 (\maxgr T - \varepsilon) > 1$. By definition of $\maxgr T$, there exist arbitrarily large $R \in \NN$ such that $\max_{v \in V} \abs{D_R(v)}^{1/R} > \maxgr T - \varepsilon$.
	
	For some such $R$, pick $v \in V$ such that $\abs{D_R(v)}^{1/R} > \maxgr T - \varepsilon$. Then for all $k \geq R$
		\[ \norm*{\sum_{j=0}^k (bQ^*)^j \1_{\{v\}} }_2 = \norm*{\sum_{j=0}^k b^j \1_{D_j(v)} }_2 \geq (b^{2R} \abs{D_R(v)})^{1/2}  > \big( b^2 (\maxgr T - \varepsilon) \big)^{R/2}. \]
	In particular,
		\[ \norm*{\sum_{j=0}^k (bQ)^j}_2 = \norm*{\sum_{j=0}^k (bQ^*)^j}_2 > \big( b^2 (\maxgr T - \varepsilon) \big)^{R/2} \]
	for all $k \geq R$ so that
		\[ \liminf_{k \to \infty} \norm*{\sum_{j=0}^k (bQ)^j}_2 \geq \big( b^2 (\maxgr T - \varepsilon) \big)^{R/2} . \]
	Since this holds for arbitrarily large $R$ and $b^2 (\maxgr T - \varepsilon) > 1$, we get the desired result.
\end{proof}
	
For the final part of Theorem \ref{thm:growthrate}, suppose $\Delta_k = O(\sqrt{\abs{V_k}})$; we show that $b^2 \bargr T \leq 1$.
For any $v \in V_k$, using that $T_k$ has diameter at most $2k$ we have
	\[ \frac{1}{\abs{V_k}} \sum_{u \in V_k} b^{d(u,v)} \geq b^{2k} . \]
Therefore
	\[ \frac{\Delta_k^2}{\abs{V_k}} \geq \sum_{v \in V_k} \frac{1}{\abs{V_k}} \sum_{u \in V_k} b^{d(u,v)} \geq b^{2k} \abs{V_k} , \]
Since $\Delta_k = O(\sqrt{\abs{V_k}})$ the previous inequality implies $b^{2} \abs{V_k}^{1/k} \leq C^{1/k}$ for large $k$, so that
	\[ b^2 \limsup_{k \to \infty} \abs{V_k}^{1/k} \leq 1. \]
By Proposition \ref{prop:altgrowthformula}, this completes the proof.

\section{Ising model and Optimality of Theorem \ref{thm:main}}
\label{sec:ising}
Let $T$ be a finite tree, and let $p \in (0,1/2]$. Consider the Markov measure $\nu$ on $\{0,1\}^V$ with uniform distribution at the root and transition matrix
	\[ P = \begin{pmatrix}
		P_{00} & P_{01} \\
		P_{10} & P_{11}
	\end{pmatrix} 
		=  \begin{pmatrix}
		1-p & p \\
		p & 1-p
	\end{pmatrix}, \]
where $P_{ji}$ denotes the probability of moving to state $i$ given that the current state is $j$, and stationary root distribution. The matrix $P$ defines a probability kernel $\kappa \colon \{0,1\} \to \Prob(\{0,1\})$ by setting $\kappa(\{i\} | j) = P_{ji}$ for $i,j \in \{0,1\}$. We take this to be the kernel at each nonroot vertex, and call this the Ising model with flip probability $p$.

The Ising model is often defined instead by defining the energy function $H \colon \{0,1\}^V \to \RR$ by
	\[ H(\sigma) = - \sum_{v,w \in V} J_{v,w} \1_{\sigma_v \ne \sigma_w} \]
(the sum is over unordered pairs) and setting $\PP(\sigma) = \frac{1}{Z} e^{-H(\sigma)}$. The quantities $J_{v,w}$ are called interaction strengths, and to match the above definition we should take 
	\[ J_{v,w} = \left\{ \begin{array}{ll} \arctanh (1-2p), & \{v,w\} \in E, \\ 0, & \text{else.} \end{array} \right. \]
	
This model is also studied with non-uniform interaction strength and with an extra contribution to $H$ called an external field. Theorem \ref{thm:main} also applies to such models, but since the goal of this section is to study the optimality of Theorem \ref{thm:main} via a model for which exact calculations are possible we restrict to the special case defined above.

Note that the uniform distribution is stationary and that by diagonalizing $P$ we can get the formula
	\[ P^n = \frac{1}{2} \begin{pmatrix}
		1+(1-2p)^n & 1-(1-2p)^n \\
		1-(1-2p)^n & 1+(1-2p)^n
	\end{pmatrix} \]
which gives the $n$-step transition probabilities.
Note also that the transition kernel $q$ has Lipschitz constant
	\[ b = \max_{x,x' \in \{0,1\}} \bar{d}\left(P(x,\cdot),P(x', \cdot)\right) = 1-2p \in [0,1) . \]
The restriction $p \leq 1/2$ ensures that we don't have to take an absolute value here, which is convenient below.

The Lipschitz constant $b$ coincides with the second-largest eigenvalue of the transition kernel, so for subperiodic trees the location of the phase transition we establish here coincides with the reconstruction threshold (see for example the survey \cite{peres2003}).
	
In this section we use probabilistic notation, letting $X$ to be a $\{0,1\}^V$-valued random variable on some probability space $(\Omega,\PP)$ with law $\nu = X_*\PP$. We write $\EE g(X) \coloneqq \int_\Omega g(X)\, d\PP$ for any measurable function $g \colon \{0,1\}^V \to \RR$. For $v \in V$, the spin at $v$ is the $v$ coordinate of $X$, which we denote $X_v$.

Let $f \colon \{0,1\}^V \to \RR$ be the 1-Lipschitz function which gives the density of ones,
	\[ f(x) = \frac{1}{\abs{V}} \# \{ v \in V \st x_v = 1\} = \frac{1}{\abs{V}} \sum_{v \in V} \1_{\{x_v = 1\}}. \]
This is often called the magnetization of $x$. By stationarity of the uniform distribution,
	\[ \EE f(X) = \tfrac{1}{2}. \]
	
The above results give sufficient conditions for $f(X)$ to concentrate around its mean along a sequence of trees; here we compare those results to what we can get by controlling the second moment.

\begin{prop}
If $f \colon \{0,1\}^V \to \RR$ is the density of ones function and the law of $X$ is the Ising model on $T$ with $p \in (0,1/2]$, then
	\[ \Var f(X) = \frac{1}{4\abs{V}^2} \sum_{(v,w) \in V^2} b^{d(v,w)}, \]
where $b = 1-2p$.
\end{prop}
\begin{proof}
We write the variance as
	\[ \Var f(X) \coloneqq \EE[f(X)^2] - [\EE f(X)]^2 = \frac{1}{\abs{V}^2} \sum_{v \in V} \sum_{w \in V} \PP(X_v = X_w = 1) - \frac{1}{4}. \]
	
Given distinct vertices $v,w \in V$, let $a = v \wedge w$ be their most recent common ancestor, and let $h_1 = d(v,a)$ and $h_2 = d(w,a)$. Then, since the spins at $v$ and $w$ are conditionally independent given $X_a$,
\begin{align*}
	\PP(X_v = X_w = 1) &= \PP(X_a = 1) \PP(X_v = 1 | X_a = 1) \PP(X_w = 1 | X_a = 1) \\
	&\qquad + \PP(X_a = 1) \PP(X_v = 1 | X_a = 0) \PP(X_w = 1 | X_a = 0) \\[0.2cm]
	&= \frac{1}{8}\left[(1+b^{h_1})(1+b^{h_2}) + (1-b^{h_1})(1-b^{h_2}) \right] \\
	&= \frac{1}{4} \big(1 + b^{h_1 + h_2} \big) \\
	&= \frac{1}{4} \big(1 + b^{d(v,w)} \big) .
\end{align*}
Inserting this expression into the above formula finishes the proof.
\end{proof}

Suppose $T$ is a fixed infinite tree of bounded degree, and for each $k$ denote the depth $k$ subtree by $T_k$.
Let $X^k$ be a random variable whose law is the Ising model on $T_k$ with flip probability $p$. By the previous proposition and the same argument as in the proof of the first part of Theorem \ref{thm:growthrate} (Section \ref{subsec:thmgrowthratepf}) we get
	\[ \lim_{k \to \infty} \Var f(X^k) = 0. \]

The preceding fact can be deduced from the above results on concentration. One new application of our variance calculation is Theorem \ref{thm:optimality}:

\subsection{Proof of Theorem \ref{thm:optimality}}
	The equivalence of the two inequalities follows from the Bobkov-G\"otze equivalence. The exponential moment bound implies
		\[ \sup \left\{ \nu \{\abs{f - \int f\, d\nu} > \varepsilon \} \st f \in \Lip_1 \big( \scrH^V \big) \right\} \leq 2 e^{-2\varepsilon^2 n^2/C^2} \quad \forall \varepsilon > 0. \]
	
	In particular, 
	\begin{align*}
		\Var f &= \int \abs{f - \int f\, d\nu}^2\, d\nu \\
		&= \int_0^\infty 2t\, \nu\big\{ \abs{f - \int f\, d\nu} > t \big\}\, dt \\
		&\leq \int_0^\infty 4t e^{-2t^2n^2/C^2}\, dt \\
		&= \frac{C^2}{n^2}.
	\end{align*}
	Applying this to the density of ones function on the Ising model as defined in the previous section, we see that
		\[ C^2 \geq \frac{1}{4} \sum_{(v,w) \in V^2} b^{d(v,w)} \geq \frac{1-b^2}{4} \Delta^2 , \]
	or, taking square roots,
		\[ C \geq \Delta \frac{\sqrt{1-b^2}}{2}. \]

\subsection{Proof of Theorem \ref{thm:bsctransition}}
	Let $\nu_k \in \Prob(\{0,1\}^{V_k})$ denote the law of the Ising model on the depth-$k$ subtree $T_k$. Suppose the sequence $\{\nu_k\}$ is a normal L\'evy family, so that there exist constants $c_1,c_2>0$ such that
		\[ \nu_k \{ f > t \} \leq c_1 e^{-c_2 \abs{V_k}t^2} \]
	for any 1-Lipschitz, expectation-zero $f$ and $t>0$. Then, as above,
	\begin{align*}
		\Var_{\nu_k} f &= \int \abs{f - \int f\, d\nu_k}^2\, d\nu_k \\
		&= \int_0^\infty 2t\, \nu_k \big\{ \abs{f} > t \big\}\, dt \\
		&\leq \int_0^\infty 4c_1 t e^{-c_2 \abs{V_k} t^2}\, dt \\
		&= \frac{2c_1}{c_2 \abs{V_k}}.
	\end{align*}
	Applied to the density of ones function on the Ising model, we get
		\[ \frac{2c_1}{c_2 \abs{V_k}} \geq \frac{1}{4 \abs{V_k}^2} \sum_{(v,w) \in V^2} b^{d(v,w)} \geq \frac{1-b^2}{4 \abs{V_k}^2} \Delta_k^2, \]
	so that $\Delta_k^2 = O(\abs{V_k})$.
		
\section{Comparison of Theorem \ref{thm:main} with related work}
\label{sec:comparison}
Kontorovich and Ramanan \cite{kontorovich2008} have proven a concentration inequality for Hamming-Lipschitz functions on finite product spaces whose form is very similar to the tail bound resulting from Theorem \ref{thm:main}; a similar result was independently obtained by Chazottes et al. \cite{chazottes2006}. While these results are in terms of mixing coefficients defined with linear-time processes in mind, in \cite{kontorovich2012} Kontorovich showed how to apply them to Markov measures indexed by finite trees. Below we state these inequalities and compare them with ours in the case of the Ising model on a finite tree.

While they both have the advantage of not requiring a process to have the Markov property, in the following we show that in relation to our Theorem \ref{thm:main}
\begin{enumerate}
	\item Kontorovich and Ramanan's inequality requires a smaller Lipschitz constant (i.e. more contractivity) in order to establish concentration for a sequence of Markov chains on trees
	\item Chazottes et al.'s inequality is sufficient to establish Theorem \ref{thm:growthrate} in the case of the Ising model.
\end{enumerate}

Let $\nu \in \Prob(\scrH^n)$ be the joint distribution of a collection of random variables $(X_1, \ldots, X_n)$ each taking values in a countable discrete metric space $\scrH$. For $1 \leq i<j \leq n$ define the mixing coefficient
	\[ \bar{\eta}_{ij} = \sup_{x_1, \ldots, x_{i-1}, x_i, x_i' \in \scrH} \begin{array}{l} \big\| \PP\big(X_j \in \cdot\ | (X_1, \ldots, X_i) = (x_1, \ldots, x_{i-1}, x_i) \big) \\[0.2em] \qquad- \ \PP\big(X_j \in \cdot\ | (X_1, \ldots, X_i) = (x_1, \ldots, x_{i-1}, x_i') \big) \big\|_{TV} . \end{array} \]
For $1 \leq i \leq n$ we set $\bar\eta_{ii} = 1$, and define $\DeltaKR$ to be the upper-triangular matrix with entries
	\[ (\DeltaKR)_{ij} = \left\{ \begin{array}{ll}
		\bar\eta_{ij} & i \leq j \\
		0 & i > j. \end{array} \right. \]
The main theorem of \cite{kontorovich2008} is that if $f \colon \scrH^n \to \RR$ is 1-Lipschitz with respect to the normalized Hamming metric we have
	\[ \nu \big\{ \abs[\big]{f - \int f\, d\nu} > \varepsilon \big\} \leq 2 e^{- n \varepsilon^2 / 2\norm{\DeltaKR}_{\infty}^2} \tag{KR}\]
where $\norm{\DeltaKR}_{\infty}$ is the $\ell^\infty$ operator norm. 

We need a way to interpret a tree-indexed process as a linear-time process in order to make sense of the mixing coefficients $\bar\eta_{ij}$ in this context. Kontorovich does this in \cite{kontorovich2012} by fixing a breadth-first ordering of the vertices: given a tree $T$ with $n$ vertices, label the vertices as $v_1, \ldots, v_{n}$ such that $v_1$ is the root and if $d(\rootvert, v_i) < d(\rootvert, v_j)$ then $i < j$. If $(X_v)_{v \in V}$ is a process indexed by $T$ we abbreviate $X_j \coloneqq X_{v_j}$.

\begin{prop}
	Let $T$ be a finite rooted tree and let $\nu$ be the joint distribution of the Ising model on $T$ with flip probability $p$. Then $\DeltaKR = \sum_{r=0}^\infty b^r Q^r$, where $Q$ is the adjacency matrix of $T$ directed away from the root as defined in Lemma \ref{lem:deltaQ}.
\end{prop}

\begin{cor}
	Let $T$ be an infinite tree, and for each $k$ let $T_k$ be the depth $k$ subtree and let $\Delta_k$, $\delta_k$, $\DeltaKR_k$ correspond to the Ising model on $T_k$ ($\DeltaKR_k$ may be induced by any breadth-first ordering of $V_k$). Then
	\begin{enumerate}
		\item $\norm{\DeltaKR_k}_\infty = \norm{\delta_k}_{\ell^\infty(V_k)}$, where $\delta_k$ is the descendant generating function of $T_k$ with $b = 1-2p$, and
		\item $\Delta_k = \norm{\DeltaKR_k \1_{V_k}}_2$. In particular $\frac{\Delta_k^2}{\abs{V_k}} \leq \norm{\DeltaKR_k}_{2}^2$, where $\norm{\DeltaKR_k}_2$ is the $\ell^2$ operator norm.
	\end{enumerate}
\end{cor}

The proofs of these statements are not completely trivial but have been omitted for the sake of brevity.

Part (1) of the previous corollary implies that $\norm{\DeltaKR_k}_\infty$ is not bounded uniformly in $k$ if
	\[ b \limsup_{k \to\infty} \abs{D_k(\rootvert)}^{1/k} = b \bargr T > 1 , \]
since in this case $\delta_k(\rootvert)$, and hence $\norm{\delta_k}_{\ell^\infty(V_k)}$, is not bounded uniformly in $k$.
Such a uniform bound is required to establish that the sequence of depth-$k$ marginals is a normal L\'evy family, so we see that the bound using $\norm{\DeltaKR}_\infty$ is unable to do so in the range $b \in (\frac{1}{\bargr T}, \frac{1}{\sqrt{\bargr T}})$. Note that even for trees regular enough that $\maxgr T = \bargr T$ this is weaker than Theorem~\ref{thm:growthrate}.

The inequality obtained by Chazottes et al. is essentially inequality (KR) above but with $\norm{\DeltaKR_k}_\infty$ replaced by $\norm{\DeltaKR_k}_2$ (the only other difference is the constant in the exponent). While part (2) of the previous corollary shows that having $\Delta_k^2/\abs{V_k}$ in the exponent of the tail bound (as we do in the present paper) is at least as effective as having $\norm{\DeltaKR_k}_2^2$, note that in the proof of the second part of Theorem \ref{thm:growthrate} we only prove $\Delta_k = O(\sqrt{\abs{V_k}})$ via the bounds
	\[ \Delta_k \leq \norm{\DeltaKR_k}_2 \sqrt{\abs{V_k}} \quad \text{and} \quad \norm{\DeltaKR_k}_2 \leq \sum_{r=0}^k b^r \norm{Q^r} = O(1) \]
	(using that $\DeltaKR = \sum_{r=0}^\infty b^r Q^r$ in the present context); therefore the inequality with $\norm{\DeltaKR_k}_2$ is sufficient to establish the second part of Theorem \ref{thm:growthrate} at least in the case of the Ising model.
	
We reiterate the remark made above, however, that this bound on $\Delta_k$ may be suboptimal; a sharper bound may yield a weaker condition than $b^2 \maxgr T < 1$ that ensures $\Delta_k = O(\sqrt{\abs{V_k}})$. On the other hand, Proposition \ref{prop:opnormbad} shows that $\norm{\DeltaKR_k}_2$ is unbounded if $b^2 \maxgr T > 1$, so that this part of Theorem \ref{thm:growthrate} cannot be improved using the inequality of Chazottes et al.

\bibliographystyle{plain}
\bibliography{references}

\begin{thebibliography}{10}

\bibitem{benjamini1994}
Itai Benjamini and Yuval Peres.
\newblock Markov {{Chains Indexed}} by {{Trees}}.
\newblock {\em The Annals of Probability}, 22(1):219--243, 1994.

\bibitem{bobkov1999}
S.G Bobkov and F~G{\"o}tze.
\newblock Exponential {{Integrability}} and {{Transportation Cost Related}} to
  {{Logarithmic Sobolev Inequalities}}.
\newblock {\em Journal of Functional Analysis}, 163(1):1--28, April 1999.

\bibitem{chazottes2006}
J.~R. Chazottes, P.~Collet, C.~K{\"u}lske, and F.~Redig.
\newblock Concentration inequalities for random fields via coupling.
\newblock {\em Probability Theory and Related Fields}, 137(1-2):201--225,
  November 2006.

\bibitem{cover2006}
T.~M. Cover and Joy~A. Thomas.
\newblock {\em Elements of Information Theory}.
\newblock {Wiley-Interscience}, {Hoboken, N.J}, 2nd ed edition, 2006.
\newblock OCLC: ocm59879802.

\bibitem{dembo2010a}
Amir Dembo and Ofer Zeitouni.
\newblock {\em Large {{Deviations Techniques}} and {{Applications}}}, volume~38
  of {\em Stochastic {{Modelling}} and {{Applied Probability}}}.
\newblock {Springer Berlin Heidelberg}, {Berlin, Heidelberg}, 2010.

\bibitem{dudley2004}
Richard~M Dudley.
\newblock {\em Real Analysis and Probability}.
\newblock {Cambridge University Press}, {Cambridge}, 2004.
\newblock OCLC: 740992059.

\bibitem{evans2000}
William Evans, Claire Kenyon, Yuval Peres, and Leonard~J. Schulman.
\newblock Broadcasting on trees and the {{Ising}} model.
\newblock {\em The Annals of Applied Probability}, 10(2):410--433, May 2000.

\bibitem{hoeffding1963}
Wassily Hoeffding.
\newblock Probability {{Inequalities}} for {{Sums}} of {{Bounded Random
  Variables}}.
\newblock {\em Journal of the American Statistical Association}, 58(301):13,
  March 1963.

\bibitem{hunter2007}
John~D. Hunter.
\newblock Matplotlib: {{A 2D Graphics Environment}}.
\newblock {\em Computing in Science \& Engineering}, 9(3):90--95, 2007.

\bibitem{kontorovich2012}
Aryeh Kontorovich.
\newblock Obtaining {{Measure Concentration}} from {{Markov Contraction}}.
\newblock {\em Markov Processes and Related Fields}, 18(4):613--638, 2012.

\bibitem{kontorovich2008}
Leonid~(Aryeh) Kontorovich and Kavita Ramanan.
\newblock Concentration inequalities for dependent random variables via the
  martingale method.
\newblock {\em The Annals of Probability}, 36(6):2126--2158, November 2008.

\bibitem{ledoux2001}
Michel Ledoux.
\newblock {\em The Concentration of Measure Phenomenon}.
\newblock Number~89 in Mathematical Surveys and Monographs. {American Math.
  Soc}, {Providence, RI}, 2001.
\newblock OCLC: 846496936.

\bibitem{lyons1990}
Russell Lyons.
\newblock Random {{Walks}} and {{Percolation}} on {{Trees}}.
\newblock {\em The Annals of Probability}, 18(3):931--958, 1990.

\bibitem{lyons2016}
Russell Lyons and Y.~Peres.
\newblock {\em Probability on Trees and Networks}.
\newblock Cambridge Series in Statistical and Probabilistic Mathematics.
  {Cambridge University Press}, {New York NY}, 2016.

\bibitem{marton1996}
K.~Marton.
\newblock Bounding $\bar{d}$-distance by informational divergence: a method to
  prove measure concentration.
\newblock {\em The Annals of Probability}, 24(2):857--866, 1996.

\bibitem{mcdiarmid1998}
Colin McDiarmid.
\newblock Concentration.
\newblock In Michel Habib, Colin McDiarmid, Jorge {Ramirez-Alfonsin}, and Bruce
  Reed, editors, {\em Probabilistic {{Methods}} for {{Algorithmic Discrete
  Mathematics}}}, pages 195--248. {Springer Berlin Heidelberg}, {Berlin,
  Heidelberg}, 1998.

\bibitem{peres2003}
Yuval Peres and Elchanan Mossel.
\newblock Information flow on trees.
\newblock {\em The Annals of Applied Probability}, 13(3):817--844, August 2003.

\bibitem{vanderwalt2011}
St{\'e}fan {van der Walt}, S~Chris Colbert, and Ga{\"e}l Varoquaux.
\newblock The {{NumPy Array}}: {{A Structure}} for {{Efficient Numerical
  Computation}}.
\newblock {\em Computing in Science \& Engineering}, 13(2):22--30, March 2011.

\end{thebibliography}

\end{document}